\documentclass[a4paper,12pt,centertags,final]{amsart}
\usepackage[T1]{fontenc}
\usepackage[utf8]{inputenc}
\usepackage{mathrsfs}
\usepackage{bbm}
\usepackage{eulervm}

\usepackage[a4paper,centering]{geometry}
\usepackage[pdfdisplaydoctitle,colorlinks,urlcolor=blue,linkcolor=blue,citecolor=blue]{hyperref}
\usepackage{mathtools}
\usepackage{xcolor}
\usepackage{mathscinet}
\usepackage{xspace}
\newtheorem{theorem}{Theorem}[section]
\newtheorem{lemma}[theorem]{Lemma}
\newtheorem{proposition}[theorem]{Proposition}
\newtheorem{corollary}[theorem]{Corollary}
\theoremstyle{definition}
\newtheorem{definition}[theorem]{Definition}

\theoremstyle{remark}
\newtheorem{remark}[theorem]{Remark}
\newtheorem{example}[theorem]{Example}
\numberwithin{equation}{section}
\newcounter{mrConstantsCounter}
\newcommand{\mrbareref}[1]{
  \ensuremath{c_\text{#1}}
}
\newcommand{\mconst}{
  \stepcounter{mrConstantsCounter}
  \mrbareref{\themrConstantsCounter}
}
\makeatletter
\newcommand{\mcdef}[1]{
  \stepcounter{mrConstantsCounter}
  \protected@write\@auxout{}{
    \string\newlabel{#1}{
      {\themrConstantsCounter}{}{}{constant.\themrConstantsCounter}{}
    }
  }
  \hypertarget{constant.\themrConstantsCounter}{\mrbareref{\ref{#1}}}
}
\makeatother
\newcommand{\mcref}[1]{
  \mrbareref{\ref{#1}}
}
\def\MRnum#1 #2\empty{#1}
\renewcommand{\MRhref}[2]{%
  \href{http://www.ams.org/mathscinet-getitem?mr=#1}{#2}
}
\renewcommand{\MR}[1]{
  \relax\ifhmode\unskip\space\fi
  \MRhref{\MRnum#1\empty}{\texttt{\Tiny[MR\MRnum#1\empty]}}
}
\newcommand{\arxiv}[1]{\href{http://arxiv.org/abs/#1}{arXiv: #1}}
\DeclareMathOperator{\supp}{Supp}
\DeclareMathOperator{\Span}{span}
\DeclareMathOperator{\Tr}{Tr}
\newcommand{\Div}[1][F]{\nabla_#1\cdot}
\newcommand{\R}{\mathbf{R}}
\newcommand{\Z}{\mathbf{Z}}
\newcommand{\E}{\mathbb{E}}
\newcommand{\Prob}{\mathbb{P}}
\newcommand{\Dc}{\mathcal{D}}
\newcommand{\Fc}{\mathcal{F}}
\newcommand{\Gc}{\mathcal{G}}
\newcommand{\Kc}{\mathcal{K}}
\newcommand{\Mac}{\mathcal{M}}
\newcommand{\Sc}{\mathcal{S}}
\newcommand{\As}{\mathscr{A}}
\newcommand{\e}{\operatorname{e}}
\newcommand{\im}{\mathrm{i}}
\newcommand{\uno}{\mathbbm{1}}
\newcommand{\cov}{\Sc}
\newcommand{\ff}{f_F}
\newcommand{\fn}{f_N}
\newcommand{\ffn}{f_{F,N}}
\newcommand{\gn}{G_N}
\newcommand{\gfn}{G_{F,N}}
\newcommand{\Fperp}{{F^{\perp_N}}}
\newcommand{\pker}{\wp^F}
\newcommand{\basis}{\mathcal{E}}
\newcommand{\eqdef}{\vcentcolon=}
\newcommand{\scalar}[1]{\langle #1 \rangle}
\newcommand{\field}[1][F]{\mathscr{#1}}
\newcommand{\loc}{{\textup{\tiny loc}}}
\newcommand{\memo}[1]{
  \ensuremath{
    \framebox{\tiny\text{\kern-2pt\textsf{#1}}\kern-2pt}
  }
  \xspace
}

\hypersetup{%
  pdftitle={Hoelder regularity of the densities for the Navier--Stokes equations with noise},
  pdfauthor={M. Romito},
  pdfcreator={M. Romito}
}
\begin{document}
  \title[H\"older regularity of the densities]
    {H\"older regularity of the densities for the Navier--Stokes equations with noise}
  \author[M. Romito]{Marco Romito}
    \address{Dipartimento di Matematica, Universit\`a di Pisa, Largo Bruno Pontecorvo 5, I--56127 Pisa, Italia}
    \email{\href{mailto:romito@dm.unipi.it}{romito@dm.unipi.it}}
    \urladdr{\url{http://www.dm.unipi.it/pages/romito}}
  \subjclass[2010]{Primary 76M35; Secondary 60H15, 60G30, 35Q30}
  \keywords{Density of laws, Navier-Stokes equations, stochastic partial
    differential equations, Besov spaces, Fokker-Planck equation.}
  \date{July 9, 2015}
  \begin{abstract}
    We prove that the densities of the finite dimensional projections of
    weak solutions of the Navier--Stokes equations driven by Gaussian noise
    are bounded and H\"older continuous, thus improving the results of
    Debussche and Romito \cite{DebRom2014}.
    
    The proof is based on analytical estimates on a conditioned
    Fokker--Planck equation solved by the density, that has a ``non--local''
    term that takes into account the influence of the rest of the
    infinite dimensional dynamics over the finite subspace under
    observation.
  \end{abstract}
\maketitle
\section{Introduction}

In this paper we improve the results of \cite{DebRom2014} for the law
of the solutions of the Navier--Stokes equations with Gaussian noise
in dimension three. We consider the problem
\begin{equation}\label{e:nse}
  \begin{cases}
    \dot u
      - \nu\Delta u
      + (u\cdot\nabla)u
      + \nabla p
      = \dot\eta,\\
    \Div[{{}}] u = 0,
  \end{cases}
\end{equation}
on the torus with periodic boundary conditions and driven by a Gaussian
noise $\dot\eta$. In the equations above $u$ is the velocity, $p$ the
pressure and $\nu$ the viscosity of an incompressible fluid. It is known
that the above problem admits global weak solutions, as well as
unique local strong solutions, as in the deterministic case.
Nevertheless the presence of noise allows to prove additional properties,
such as continuous dependence on initial data
\cite{DapDeb2003,DebOda2006,FlaRom2006,FlaRom2007,FlaRom2008}, as well
as convergence to equilibrium \cite{Oda2007,Rom2008}.
See also the recent surveys \cite{FlaRom2008,Deb2013} for a general
introduction to the problem.
\medskip

Our interest in the existence of densities stems from a series
of mathematical motivations. The first and foremost is the
investigation of the regularity properties of solutions of the
Navier--Stokes equations. We study here the regularity properties
of densities associated to the probabilistic distribution of the
solution, as existence and regularity of densities can be seen as
a different type of regularity.

Our results concerns the existence of densities for suitable finite
dimensional projections of the solutions of \eqref{e:nse},
and one reason for this is that in infinite dimension
there is no standard reference measure (as is the Lebesgue measure
in finite dimension), any choice should be necessary tailored to the
problem at hand, and in our case we do not know enough of
the problem.

An interesting difficulty in proving regularity of the densities
emerges as a by--product of the more general and fundamental
problem of proving uniqueness and regularity of solutions of the
Navier--Stokes equations. Indeed, a classical tool is the
Malliavin calculus, and it is easy to be convinced that
it is not available here. Indeed, formally, the
equation satisfied by the Malliavin derivative is the linearisation
of \eqref{e:nse} and thus, estimates on the linearized equation
are as good for the density as for uniqueness. We remark that
in the case of the two dimensional Navier-Stokes equation, existence
and smoothness of densities for the finite dimensional projections
of the solutions are proved in \cite{MatPar2006} with Malliavin 
calculus.

This settles the need of methods to prove existence and regularity
of the density that do not rely on this calculus, as done
in~\cite{DebRom2014}. For other works in this direction, see for instance
\cite{Dem2011,BalCar2014p,KohTan2012,HayKohYuk2013,HayKohYuk2014}.
\medskip

Existence of densities and their regularity in Besov spaces has been
proved in~\cite{DebRom2014} (see also \cite{Rom2013a,Rom2014a,Rom2014b}),
by extending and generalising a one dimensional idea from~\cite{FouPri2010}.
Time regularity of the density has been proved in~\cite{Rom2014pb}.
The method introduced in~\cite{DebRom2014} is simple but effective and
has been already used in other problems (see for
instance~\cite{DebFou2013,Fou2015,SanSus2015,SanSus2015p}) to prove
the existence of densities.

The results of~\cite{DebRom2014} ensure that the density of the
projection of solution at some fixed time on some finite dimensional
sub--space is in the Besov space $B^{1-}_{1,\infty}$. Roughly speaking,
this says that densities have (almost) integrable
derivative (see Section~\ref{s:notations} for a short introduction
to Besov spaces).
\medskip

In this paper we show a proof of H\"older regularity of densities of
finite dimensional projections that is completely analytic and, unlike
\cite{DebRom2014}, does not rely on probabilistic ideas. In fact
we follow a classical approach to existence and regularity
of densities, namely the Fokker--Planck equation.
The Fokker--Planck equation describes the
evolution of the density of the It\=o process solution of a stochastic
equation. Here we only look at a partial information on the solution
(namely, a finite dimensional projection), thus we derive in
Section~\ref{s:formulation} a Fokker--Planck equation with
a ``non--local'' term that takes into account the effect of the
dynamics outside the finite--dimensional space under observation.
The non--local term is indeed a conditional expectation and its
regularity is known only in terms of the unknown density itself.
This makes our Fokker--Planck equation slightly non--standard.
We re--derive in this framework the results of~\cite{DebRom2014}
(see Proposition~\ref{p:bbesov}), we then prove the core
result of the paper, namely boundedness of the densities,
in Proposition~\ref{p:bounded}, and finally the H\"older regularity.

Our proof of boundedness requires that, at least at the level of
the Galerkin approximations we work with, we already know that the
densities are bounded, possibly with bounds depending on the
approximation (and so useless for the limiting problem). We derive
these bounds on the approximations in the appendix by standard results
for hypo--elliptic diffusions and to do so we need to assume that the
noise is ``sufficiently non--degenerate'' (see Section~\ref{s:assumptions}).
With periodic boundary conditions the problem has been already
thoroughly analyzed in~\cite{Rom2004} and this is the reason we
mainly focus on the problem on the torus. There is in principle
no limitation for the problem with Dirichlet boundary conditions,
once smoothness of the densities at the level of approximations
is settled.
\medskip

We believe that this is just a technical requirement (indeed it is not
needed in~\cite{DebRom2014}) that depends on the approach we have
followed. Moreover, there is an inherent limitation in the
Fokker--Planck approach, in that it is less flexible than the method
developed in~\cite{DebRom2014} and thus cannot be used, in general,
to evaluate the density of quantities that do not have an associated
evolution equation, as for instance in~\cite{SanSus2015,SanSus2015p},
as well as for nonlinear functions of the solution of a diffusion
process. The development of a probabilistic proof of the results
of this paper via a generalization of~\cite{DebRom2014}  is
currently the subject of an on--going work.
\section{Main result}
\subsection{Notations}\label{s:notations}

We shall use the following notations. If $K$ is an Hilbert space,
and $F\subset K$ a subspace, we denote by $\pi_F:K\to K$
the orthogonal projection of $K$ onto $F$, and by
$\Span[x_1,\dots,x_n]$ the subspace of $K$ spanned by
its elements $x_1,\dots,x_n$. Given a linear operator
$\mathcal{Q}:K\to K'$, we denote by $\mathcal{Q}^\star$
its adjoint.

We recall the definition of Besov spaces. Given $f:\R^d\to\R$,
define
\[
  \begin{gathered}
    (\Delta_h^1f)(x)
      \eqdef f(x+h)-f(x),\\
    (\Delta_h^nf)(x)
      \eqdef\Delta_h^1(\Delta_h^{n-1}f)(x)
      = \sum_{j=0}^n (-1)^{n-j}\binom{n}{j} f(x+jh),
  \end{gathered}
\]
and, for $s>0$, $1\leq p\leq\infty$, $1\leq q<\infty$,
\[
  [f]_{B_{p,q}^s}
    \eqdef \Bigl(\int_{\{|h|\leq 1\}}\frac{\|\Delta_h^n f\|_{L^p}^q}{|h|^{sq}}
          \frac{dh}{|h|^d}\Bigr)^{\frac1q},
\]
and for $q=\infty$,
\[
  [f]_{B_{p,\infty}^s}
    \eqdef \sup_{|h|\leq 1}\frac{\|\Delta_h^n f\|_{L^p}}{|h|^s},
\]
where $n$ is any integer strictly larger than $s$ (the above
semi--norms are independent of the choice of $n$, as long as
$n>s$).
Given $s>0$, $1\leq p\leq\infty$ and $1\leq q\leq\infty$,
define
\[
  B_{p,q}^s(\R^d)
    \eqdef \{f: \|f\|_{L^p} + [f]_{B_{p,q}^s}<\infty\}.
\]
This is a Banach space when endowed with the norm
$\|f\|_{B_{p,q}^s} := \|f\|_{L^p} + [f]_{B_{p,q}^s}$.

When in particular $p=q=\infty$ and $s\in(0,1)$,
the Besov space $B_{\infty,\infty}^s(\R^d)$
coincides with the H\"older space $C^s_b(\R^d)$,
and in that case we will denote by $\|\cdot\|_{C^s_b}$
and $[\cdot]_{C^s_b}$ the corresponding norm and
semi--norm. Notice that a more general definition of Besov
spaces, that includes also the case $s\leq0$, is based
on the Littlewood--Paley decomposition. We refer
to~\cite{Tri1983,Tri1992} for more details on the definitions
we have given and the connection with the general definition.

For a Hilbert space $K$ and a finite dimensional sub--space
$F$ of $K$, we shall denote by $L^p(F)$, $C^s_b(F)$
and $B^s_{p,q}(F)$ the Lebesgue, H\"older and
Besov spaces, respectively, on $F$, when $F$ is
identified with a Euclidean space.

\subsection{The Navier--Stokes equations}\label{s:nse}

Let $H$ be the standard space of periodic square summable divergence
free vector fields, defined as the closure of periodic divergence free
smooth vector fields with zero spatial mean,
with inner product $\scalar{\cdot,\cdot}_H$ and norm $\|\cdot\|_H$.
Define likewise $V$ as the closure of the same space of
functions with respect to the $H^1$ norm.
Let $\Pi_L$ be the Leray projector, and let $A=-\Pi_L\Delta$ be the
Stokes operator.

In view of Section~\ref{s:assumptions} we specify in details
an orthonormal basis of $H$. Let $\Z^3_\star=\Z^3\setminus\{0\}$
and consider for every $k\in\Z^3_\star$ an orthonormal basis $x_k^1,x_k^2$
of the subspace $k^\perp$ of $\R^3$ orthogonal to the vector $k$. Choose
moreover $x_k^1,x_k^2$ so that $x_{-k}^i=x_k^i$, $i=1,2$. An orthonormal
basis of $H$ is given by $\basis=(e_k^i)_{i=1,2,k\in\Z^3_\star}$,
where $e_k^i=x_k^i\e^{\im k \cdot x}$. Clearly $\basis$ is a basis
of eigenvectors of the Stokes operator $A$, and different choices of
$(x_k^i)_{i=1,2, k\in\Z^3_\star}$ yield different bases. 

Define the bi--linear operator $B:V\times V\to V'$
as $B(u,v) = \Pi_L\left( u\cdot\nabla v\right)$, $u,v\in V$,
and recall that $\scalar{u_1, B(u_2,u_3)}=-\scalar{u_3, B(u_2,u_1)}$.
We will use the shorthand $B(u)$ for $B(u,u)$.
We refer to Temam \cite{Tem1995} for a detailed account of all
the above definitions.

Let $\cov:H\to H$ be a Hilbert--Schmidt operator, then with
the above notations, we can consider the following Navier--Stokes
equations 
\begin{equation}\label{e:nseabs}
  du + (\nu Au + B(u))\,dt = \cov\,dW,
\end{equation}
with initial condition $u(0)=x\in H$, where $W$ is a cylindrical
Wiener process (see \cite{DapZab1992} for further details) in $H$.
It is well--known \cite{Fla2008}
that for every $x\in H$ there exists a martingale solution of this equation,
that is a filtered probability space $(\smash{\widetilde\Omega},
\smash{\widetilde{\field}},\smash{\widetilde{\Prob}},
(\smash{\widetilde{\field}}_t)_{t\geq 0})$,
a cylindrical Wiener process $\smash{\widetilde W}$ and a process $u$
with trajectories
in  $C([0,\infty);D(A)')\cap L^\infty_\loc([0,\infty),H)\cap
L^2_\loc([0,\infty);V)$
adapted to $(\smash{\widetilde{\field}}_t)_{t\geq 0}$ such that the above
equation is satisfied with $\smash{\widetilde W}$ replacing $W$.
\subsubsection{Galerkin approximations}

Given an integer $N\geq 1$, denote by $H_N$ the sub--space
$H_N = \Span[e_k^i:\ i=1,2,\ |k|\leq N]$ and denote by $\pi_N = \pi_{H_N}$
the projection onto $H_N$. It is standard (see for instance
\cite{Fla2008}) to verify that the problem
\begin{equation}\label{e:galerkin}
  du^N + \bigl(\nu Au^N + B^N(u^N))\,dt = \pi_N\cov\,dW,
\end{equation}
where $B^N(\cdot) = \pi_N B(\pi_N\cdot)$, admits a unique strong
solution $u^N$ for every initial condition $x^N\in H_N$.
The proof is a based on the simple fact that 
in finite dimension all norms are equivalent, thus
given any finite dimensional sub--space $F$ of $H$, there is
$\mcdef{cc:finite}>0$ such that
\[
  \|\pi_FAx\|_H
    \leq\mcref{cc:finite}\|x\|_H,
      \qquad
  \|\pi_FB(x_1,x_2)\|_H
    \leq\mcref{cc:finite}\|x_1\|_H\|x_2\|_H.
\]
If $x\in H$, $x^N = \pi_N x$ and $\Prob^N_x$ is the distribution
of the solution of the problem above with initial condition $x^N$,
then any limit point of $(\Prob^N_x)_{N\geq1}$ is a solution
of the martingale problem associated to \eqref{e:nseabs} with initial
condition $x$. In the rest of the paper we will consider only
solutions of \eqref{e:nseabs} of this type, as specified by
the following definition.
\begin{definition}\label{d:sol}
  A solution of \eqref{e:nseabs} with initial condition $x\in H$
  is any process $u$ with $u(0)=x$ and with trajectories in
  $C([0,\infty);D(A)')\cap L^\infty_\loc([0,\infty),H)\cap
  L^2_\loc([0,\infty);V)$, such that there are a sequence of integers
  $N_k\uparrow\infty$ and a sequence $x_k\in H_{N_k}$ such that
  $x_k\to u(0)$, and $u$ is a limit point, in distributions,
  of $(u^{N_k})_{k\geq1}$, where $u^N$ is solution of \eqref{e:galerkin}.
\end{definition}
\subsection{Assumptions}\label{s:assumptions}

Consider again the Hilbert--Schmidt operator $\cov:H\to H$.
We will assume
\begin{equation}\label{a:diagonal}
  \colorbox{lightgray}{
    \begin{minipage}{.8\linewidth}
      $\cov$ is diagonal in a basis $\basis$,
    \end{minipage}
  }
\end{equation}
where bases $\basis$ have been defined in Section~\ref{s:nse}.
Under this assumption the operators $A$, $\cov$ commute, $\basis$
is also a orthonormal basis of eigenvectors of the covariance
$\cov\cov^\star$, and the noise $\cov\dot W$ is homogeneous in space.
This assumption is taken for the sake of simplicity
(See Remark~\ref{r:nostokes} below), in view of applying the
results of~\cite{Rom2004}.

We shall also need the following global non--degeneracy condition,
\begin{equation}\label{a:nondegenerate}
  \colorbox{lightgray}{
    \begin{minipage}{.8\linewidth}
      let $\Kc = \{k: \scalar{\cov e_k^i, e_k^i}>0, i=1,2\}$,
      then $\Kc$ is an algebraic system of generators of the
      group $(\Z^3,+)$.
    \end{minipage}
  }
\end{equation}
This assumption has been introduced in~\cite{Rom2004} to ensure
that Galerkin approximations \eqref{e:galerkin}, with
$N$ large enough, are hypo--elliptic diffusions. We will use
this fact to deduce, see Theorem~\ref{t:schwartz}, that
the solution of \eqref{e:galerkin} has a smooth density
with respect to the Lebesgue measure on $H_N$. Notice
that the assumption essentially requires that the noise
is stirring all directions in $H$, albeit indirectly.
We believe this assumption is technical and depends on
the way we have proved our results, namely to ensure that
the computation in the next section are rigorous. As already
mentioned in the introduction, a probabilistic proof of
our main results would definitely get rid of this assumption.

We will consider densities for the projections of solutions
of \eqref{e:nseabs} over finite dimensional sub--spaces
of $H$. Given a finite dimensional subspace $F$ of $H$,
we consider the following conditions,
\begin{equation}\label{a:finitespan}
  \colorbox{lightgray}{
    \begin{minipage}{.8\linewidth}
      $F$ is the span of a finite subset of $\basis$,
    \end{minipage}
  }
\end{equation}
where the basis $\basis$ is the same of assumption \eqref{a:diagonal},
and that the noise is non--degenerate on $F$, namely,
\begin{equation}\label{a:Fnondegenerate}
  \colorbox{lightgray}{
    \begin{minipage}{.8\linewidth}
      $\pi_F\cov\cov^\star\pi_F$ is a non--singular matrix.
    \end{minipage}
  }
\end{equation}
This condition has been introduced in~\cite{DebRom2014,Rom2014pb},
and amounts to say that the covariance has full range in $F$ or,
in different words, that the noise is \emph{directly} stirring all
directions of $F$.
\begin{remark}\label{r:nostokes}
  In general, there is nothing special with the bases $\basis$
  provided by the eigenvectors of the Stokes operator and our
  results would work when applied to Galerkin approximations
  generated by any (smooth enough) orthonormal basis of $H$.
  On the other hand we are using the results of~\cite{Rom2004}
  and the setting with the bases $\basis$ is the most suitable.
  
  One could rather work with a general basis of eigenvectors
  and assume that the spectral Galerkin approximations are
  hypo--elliptic diffusions, thus ensuring the conclusions of
  Theorems~\ref{t:schwartz} and~\ref{t:malliavin}.
  We have preferred to proceed with the explicit version
  of the assumptions. Similar considerations apply for
  the problem on a bounded domain with Dirichlet boundary
  conditions.
\end{remark}
\begin{remark}
  The two non--degeneracy assumptions given above are, in a way,
  independent. An easy example where assumption \eqref{a:nondegenerate}
  holds while assumption \eqref{a:Fnondegenerate} does not is provided
  in~\cite{RomXu2011}, where all but the ``low modes'' are forced
  by the noise and \eqref{a:nondegenerate} holds. But if $F$ contains
  low modes components, \eqref{a:Fnondegenerate} is not true.
  
  On the other hand, fix $k_0\in\Z^3_\star$ and set
  \[
    \cov
      = \sum_{n=1}^\infty \sigma_n^1\scalar{\cdot,e_{nk_0}^1}e_{nk_0}^1
          + \sigma_n^2\scalar{\cdot,e_{nk_0}^2}e_{nk_0}^2,
  \]
  so that the set $\Kc$ introduced in assumption \eqref{a:nondegenerate}
  is $\Kc=\{nk_0:n\in\Z,n\neq0\}$. It is easy to check that, formally,
  the Navier--Stokes dynamics \eqref{e:nseabs} is closed in the subspace
  $\Span[e_k^i: i=1,2,k\in\Kc]$, and in particular is not hypo--elliptic.
  If $F$ is spanned by elements of $\Kc$, then \eqref{a:Fnondegenerate}
  holds.
  
  We do believe (and this is the subject of a work in progress) that also
  in this case the projections on $F$  have H\"older densities. Notice
  that Proposition~\ref{p:bbesov} below only ensures that these densities
  are in $B^1_{1,\infty}$.
\end{remark}
\subsection{Main result}

The main result of the paper is as follows.
\begin{theorem}\label{t:main}
  Consider equation \eqref{e:nseabs} and assume $\cov$
  satisfies assumptions \eqref{a:diagonal} and
  \eqref{a:nondegenerate}. Let $u$ be a solution of \eqref{e:nseabs}
  as in Definition~\ref{d:sol} and let $F$ be a finite dimensional
  subspace of $H$ satisfying conditions \eqref{a:finitespan}
  and \eqref{a:Fnondegenerate}. For every $t>0$ denote by $\ff(t)$
  the density of $\pi_F u(t)$ with respect to the Lebesgue measure
  on $F$. Then for every $T>0$ and every $\alpha\in(0,1)$,
  \[
    \sup_{(0,T)}t^{\frac12(d+\alpha)}\|\ff(t)\|_{C^\alpha_b}
      <\infty,
  \]
  where $d=\dim F$.
\end{theorem}
In particular the density $\ff$ is bounded and H\"older continuous
for every exponent $\alpha<1$. The proof of this theorem will
be given at the end of Section~\ref{s:holder}.
\section{Formulation of the conditioned Fokker--Planck equation}%
  \label{s:formulation}

Under our assumptions (see Theorem~\ref{t:schwartz}) we know that, if
$u^N$ is a solution of \eqref{e:galerkin}, then for every $t>0$ the law
of $u^N(t)$ has a smooth density $\fn(t)$ with respect to the Lebesgue
measure in the Schwartz space. Then it is a standard fact that $\fn$
satisfies a Fokker--Planck equation, that in our notations reads
\begin{equation}\label{e:FKN}
  \partial_t\fn
    = \frac12\As_N\fn
      + \Div[N]\bigl((\nu A_Nx+B_N(x))\fn\bigr),
\end{equation}
where
\[ 
  \As_N g
    = \Tr(\cov_N\cov_N^\star D^2g),
\]
and $\Div[N]$ is the divergence on $H_N$.

Fix a subspace $F$ of $H_N$ such that conditions \eqref{a:finitespan}
and \eqref{a:Fnondegenerate} hold
and consider the projection $\pi_F u^N$ of $u^N$ on $F$. The
marginal density is given by
\[
  \ffn(t,x')
    = \int_{\Fperp} \fn(t,x'+x'')\,dx'',
\]
where $\Fperp$ is the space orthogonal to F in $H_N$. In the sequel
we will understand $x\in H_N$ as $x=(x',x'')$ with $x'\in F$ and
$x''\in\Fperp$. We wish to derive now an equation satisfied by $\ffn$.
By integrating the equation \eqref{e:FKN} over $\Fperp$, it is
not difficult to see that
\begin{equation}\label{e:FKF}
  \partial_t\ffn
    = \frac12\As_F\ffn
      + \Div[F](\gfn\ffn),
\end{equation}
where $\As_F=\pi_F\As_N\pi_F$ (the resulting operator
is independent of $N$ due to the assumption \eqref{a:diagonal}
that the covariance is diagonal in the Galerkin basis), and
\begin{equation}\label{e:G}
  \gfn(t,x')
    = \nu\pi_FAx'
      + \E[\pi_F B_N(u^N(t))|\pi_Fu^N(t)=x'].
\end{equation}
Indeed, set for brevity $\gn(x)=\nu A_Nx+B_N(x)$, then
\[
  \int_{\Fperp}\Div[N](\gn\fn)
    = \int_{\Fperp}\Div(\pi_F\gn\fn)
      + \int_{\Fperp}\Div[\Fperp](\pi_\Fperp \gn\fn)
\]
where the second integral in the displayed formula above
is zero by integration by parts, and the first integral
can be reinterpreted as
\[
  \begin{aligned}
    \int_{\Fperp}\Div(\pi_F\gn\fn)\,dx''
      &= \int_{\Fperp}\Div\bigl(\pi_F\gn f_{\Fperp|F}(t,x|x')\ffn(t,x')\bigr)\,dx''\\
      &= \Div\bigl(\E[\pi_F\gn(u(t))|\pi_Fu(t)=x']\ffn(t,x')\bigr).
  \end{aligned}
\]
Here $f_{\Fperp|F}(t,x|x')$ is the conditional density
of $u^N(t)$ given $\pi_Fu^N(t)$, and
$\Div$ is the divergence on $F$. An additional simplification,
due to the fact that $A_N$ is also diagonal in the Galerkin
basis, yields \eqref{e:G}. Similar but simpler computations
show that the contribution of $\As_N\fn$, when averaged
over $\Fperp$, is the term $\As_F\ffn$.
\begin{lemma}\label{l:Gbound}
  Let $N$ be large enough (that $F\subset H_N$)
  and $u^N$ be a solution of \eqref{e:galerkin}.
  Then $\gfn(t)\in L^p(F;\ffn\,dx')$ for every $p\geq1$ and $t>0$.
  Moreover, for every $T>0$ there is $\mcdef{cc:Gbound}
  =\mcref{cc:Gbound}(p,\nu,T,F,\|u^N(0)\|_H)>0$ such that
  \begin{equation}\label{e:Gbound}
    \Gc_p(T)^p
      \eqdef\sup_{[0,T]}\int_F|\gfn(t,x')|^p\ffn(t,x')\,dx'
      \leq\mcref{cc:Gbound}.
  \end{equation}
\end{lemma}
\begin{proof}
  Fix $p\geq1$ and $t>0$. Let $f_{\Fperp|F}(t,x''|x')$ be
  the conditional density of $u^N(t)$ given $\pi_F u^N(t)$,
  then by the H\"older inequality,
  \[
    |\gfn(t,x')|^p
      \leq \int_{\Fperp} |\pi_F \gn(x',x'')|^p f_{\Fperp|F}(t,x''|x')\,dx'',
  \]
  hence 
  \[
    \begin{aligned}
      \int_F |\gfn(t,x')|^p \ffn(t,x')\,dx'
        &\leq \int_{H_N} |\pi_F \gn(x)|^p \fn(t,x)\,dx\\
        &= \E[\|\pi_F \gn(u^N(t))\|_H^p]\\
        &\leq \mconst\bigl(\nu^p\E[\|u^N(t)\|_H^p]
              + \E[\|\pi_FB_N(u^N(t))\|_H^p]\bigr)\\
        &\leq \mconst\bigl(\nu^p\E[\|u^N(t)\|_H^p]
                + \E[\|u^N(t)\|_H^{2p}]\bigr)\\
        &\leq \mcref{cc:Gbound},
    \end{aligned}
  \]
  where we have used the fact that on $F$ all norms are equivalent
  and the (uniform in $N$) estimate \eqref{e:polymoment}.
\end{proof}
\begin{remark}[exponential bound for $\gfn$]
  A slightly (although useless so far) better estimate can be obtained
  using \eqref{e:expomoment}. By the Jensen inequality
  \[
    \e^{\lambda|\gfn(t,x')|}
      \leq\int_{\Fperp}\e^{\lambda|\pi_F \gn(x)}f_{\Fperp|F}(t,x''|x')\,dx'',
  \]
  hence, as in the proof of the lemma above, and by \eqref{e:expomoment},
  \[
    \int_F \e^{\lambda|\gfn(t,x')|}\ffn(t,x')\,dx'
      \leq\E[\e^{\lambda\|\pi_F \gn(u^N(t))\|_H}]
      \leq\E[\e^{\lambda\mconst(1 + \|u^N(t)\|_H^2)}],
  \]
  that is finite for $\lambda$ small enough.
\end{remark}
\begin{remark}
  The equation \eqref{e:FKF} can be recast in a way that shows more
  explicitly how the contribution of modes in $\Fperp$ enter in the 
  evolution of $\ffn$. More precisely, for $x\in H_N$,
  \[
    \pi_F B_N(x)
      = \pi_F B_N(x')
          + \pi_F B_N(x',x'')
          + \pi_F B_N(x'',x')
          + \pi_F B_N(x''),
  \]
  and so
  \[
    \begin{aligned}
      \gfn(t,x')
        &= \nu A_Fx'
           + B_F(x')
           + \E[\pi_F B_N(\pi_Fu^N(t),\pi_{\Fperp}u^N(t))\,|\,\pi_Fu^N(t)=x']\\
        &\quad+ \E[\pi_F B_N(\pi_{\Fperp}u^N(t),\pi_Fu^N(t))\,|\,\pi_Fu^N(t)=x']\\
        &\quad+ \E[\pi_F B_N(\pi_{\Fperp}u^N(t))\,|\,\pi_Fu^N(t)=x'],
    \end{aligned}
  \]
  that is $\ffn$ solves an equation analogous to \eqref{e:FKN} with an additional
  term that takes into account the influence of the evolution of modes from
  $\Fperp$.
\end{remark}
\subsection{Integral formulation of \texorpdfstring{\eqref{e:FKF}}{(3.2)}}

Due to our assumptions \eqref{a:finitespan} and \eqref{a:Fnondegenerate}
on $F$, the operator
$\As_F$ is elliptic and with constant coefficients.
Let $\pker$ be its kernel, then the integral formulation
of equation~\eqref{e:FKF} is given by,
\begin{equation}\label{e:intfkf}
  \ffn(t,x')
    = \pker_t(x'-x_0')
      - \int_0^t (\nabla_F\pker_{t-s}\star\bigl(\gfn(s,\cdot)\ffn(s,\cdot)\bigr)(x')\,ds,
\end{equation}
where $x_0'=\pi_f u^N(0)$.
\begin{proposition}\label{p:preg}
  Under assumptions \eqref{a:finitespan} and \eqref{a:Fnondegenerate}
  on $F$, for every $p$ with $1\leq p\leq\infty$ and every $n\geq1$
  there is $\mcdef{cc:p}>0$ such that for every $h\in F$
  with $|h|\leq1$, and $t>0$,
  \begin{alignat}{2}
    \|\pker_t\|_{L^p}
      &\leq \mcref{cc:p}t^{-\frac{d}{2q}}
        \qquad&
    \|\nabla\pker_t\|_{L^p}
      &\leq \mcref{cc:p}t^{-(\frac{d}{2q} + \frac12)},\\
    \|\Delta_h^n\pker_t\|_{L^p}
      &\leq \mcref{cc:p}t^{-\frac{d}{2q}}\Bigl(1\wedge\frac{|h|}{\sqrt{t}}\Bigr)^n,
        \qquad&
    \|\Delta_h^n\nabla\pker_t\|_{L^p}
      &\leq \mcref{cc:p}t^{-(\frac{d}{2q} + \frac12)}\Bigl(1\wedge\frac{|h|}{\sqrt{t}}\Bigr)^n,
  \end{alignat}
  where $q$ is the H\"older conjugate exponent of $p$ and $d=\dim F$.
\end{proposition}
\begin{proof}
  The operator $\As_F$ is a second order elliptic constant
  coefficients operator, so its kernel $\pker_t$ is, up to an
  invertible linear
  change of variables, the standard heat kernel. In particular, 
  by parabolic scaling, $\pker_t(x) = t^{-d/2}\pker_1(x/\sqrt{t})$.
  With these observation at hand the proof of the proposition
  follows from similar computations for the heat kernel. The
  computations are elementary and thus omitted.
\end{proof}
\section{H\"older regularity}\label{s:holder}

In this section we prove our main Theorem~\ref{t:main}.
Prior to this, we give a proof of the result of \cite{DebRom2014}
using the Fokker--Planck formulation (see Proposition~\ref{p:bbesov}).
Notice that, consistently with \cite{DebRom2014}, we do not need
assumption \eqref{a:nondegenerate} to do so. Then, as an intermediate step,
we prove boundedness of the densities. This is, not surprisingly in fact,
the crucial step and with boundedness at hand the H\"older regularity
follows easily. In order to obtain $L^\infty$ bounds we need to know
though that densities are smooth (albeit with bounds not necessarily
uniform in $N$). We will take care of this in the appendix.

Notice finally that, unlike in \cite{DebRom2014}, we have not been
able to derive better bounds for stationary solutions. This may have
a twofold reason. On the one hand the higher summability needed to
derive better bounds for stationary solutions are not as good as those
of Lemma~\ref{l:Gbound}. On the other hand the method of \cite{DebRom2014}
exploits non--trivially correlations in time that are harder to use in
this framework.
\subsection{Basic Besov regularity}

In this section we wish to give a proof of the Besov
regularity of the density $\ff$ alternative to \cite{DebRom2014}, using
the formulation with the conditioned Fokker--Planck \eqref{e:intfkf}.
This is to emphasize that in \cite{DebRom2014} Lemma~\ref{l:Gbound}
is (implicitly) used only with $p=1$ and thus that  there is space for
improvement. This will be then the subject of the rest of the section. 
\begin{proposition}\label{p:bbesov}
  Let $F$ be a subspace of $H$ satisfying \eqref{a:finitespan}
  and \eqref{a:Fnondegenerate}. Then
  there is $\mcdef{cc:bbesov}>0$ such that for every $N\geq1$,
  $x_0\in H_N$, and $t>0$
  \[
    \|\ffn(t)\|_{B^1_{1,\infty}}
      \leq\frac{\mcref{cc:bbesov}}{\sqrt{1\wedge t}}
        \bigl(1 + \Gc_1(t)\bigr),
  \]
  where $\ffn$ is the density of the solution of \eqref{e:galerkin}
  with initial condition $x_0$.
\end{proposition}
\begin{proof}
  Clearly
  $\|\ffn\|_{L^1}=1$, so we need to estimate only the Besov
  semi--norm $[\ffn]_{B^1_{1,\infty}}$. Let $h\in F$, with
  $|h|_F\leq 1$, then by Proposition~\ref{p:preg}
  \[
    \|\Delta_h^2\pker_t\|_{L^1}
      \leq\mcref{cc:p}\Bigl(1\wedge\frac{|h|}{\sqrt{t}}\Bigr)
      \leq\frac{\mcref{cc:p}}{\sqrt{1\wedge t}}|h|,
        \qquad
    \|\Delta_h^2\nabla\pker_t\|_{L^1}
      \leq\frac{\mcref{cc:p}}{\sqrt{t}}
        \Bigl(1\wedge\frac{|h|}{\sqrt{t}}\Bigr)^2.
  \]
  We have by \eqref{e:intfkf} that
  \[
    \Delta_h^2\ffn(t)
      = \Delta_h^2\pker_t
        - \int_0^t (\Delta_h^2\nabla\pker_{t-s})\star(\gfn(s)\ffn(s))\,ds,
  \]
  and, by the H\"older inequality and Lemma~\ref{l:Gbound},
  for every $s\leq t$,
  \[
    \begin{aligned}
      \|(\Delta_h^2\nabla\pker_{t-s})\star(\gfn(s)\ffn(s))\|_{L^1}
        &\leq \|\Delta_h^2\nabla\pker_{t-s}\|_{L^1}\|\gfn(s)\ffn(s)\|_{L^1}\\
        &\leq \frac{\mcref{cc:p}}{\sqrt{t-s}}
           \Bigl(1\wedge\frac{|h|}{\sqrt{t-s}}\Bigr)^2
           \Gc_1(t),
    \end{aligned}
  \]
  where $\Gc_1$ is defined in \eqref{e:Gbound}. Therefore,
  \[
    \begin{aligned}
      \|\Delta_h^2\ffn(t)\|_{L^1}
        &\leq \frac{\mcref{cc:p}}{\sqrt{1\wedge t}}|h|
           + \Gc_1(t)\int_0^t \frac{\mcref{cc:p}}{\sqrt{t-s}}
             \Bigl(1\wedge\frac{|h|}{\sqrt{t-s}}\Bigr)^2\,ds\\
        &\leq\frac{4\mcref{cc:p}}{\sqrt{1\wedge t}}(1 + \Gc_1(t))|h|,
    \end{aligned}
  \]
  and, by definition, $[\ffn]_{B^1_{1,\infty}}\leq 4\mcref{cc:p}
  (1\wedge t)^{-1/2}(1 + \Gc_1(t))$.
\end{proof}
In the above proof we have only used that $\gfn\in L^1(F)$.
As long as we only know this, the previous result is essentially optimal.
Indeed, the term $\Div(\gfn\ffn)$ is, roughly, in $W^{-1,1}$ hence,
by convolution with the heat kernel, we can expect that $\ffn$
is at most in $W^{1,1}$ (that is very close to what we
have proved).

Since by Lemma~\ref{l:Gbound} the quantity $\Gc_1$ is uniformly
bounded in $N$, in the limit $N\to\infty$ we can derive a result
for solutions of the infinite dimensional problem,
as in~\cite{DebRom2014}.
\begin{corollary}\label{c:bbesov}
  Given a weak martingale solution $u$ of the Navier--Stokes equations
  as in Definition~\ref{d:sol} and a finite dimensional sub--space $F$ of $H$
  such that conditions \eqref{a:finitespan} and \eqref{a:Fnondegenerate}
  hold, for every $t>0$ the random variable $\pi_F u(t)$ admits a density
  in $B^1_{1,\infty}(F)$ with respect to the Lebesgue measure on $F$.
\end{corollary}
\begin{proof}
  Let $u$ be a solution of \eqref{e:nseabs} according to
  Definition~\ref{d:sol}. Then $u$ is a limit point,
  in distribution, of the sequence $(u^N)_{N\geq1}$
  of solutions of \eqref{e:galerkin}. By
  Proposition~\ref{p:bbesov} and the embedding of
  $B^1_{1,\infty}$ in $L^p$ (for a $p>1$ depending on
  the dimension of $F$), the densities $\ffn(t)$
  of $\pi_Fu^N(t)$ are uniformly integrable.
  We can thus find sub--sequences, that we will
  keep denoting by $(u^N)_{N\geq1}$ and $(\ffn)_{N\geq1}$,
  such that $u^N$ converges in distribution to $u$ and
  $\ffn(t)$ converges weakly in $L^1$ to the density $\ff(t)$
  of $\pi_Fu(t)$. The $B^1_{1,\infty}$ bound on $\ff$
  follows now easily because of the weak convergence.
\end{proof}
\subsection{Boundedness of the densities}

As an intermediate step in the proof of our main result
we prove that under our assumptions the densities are
bounded. This is the crucial step and having boundedness
at hand, H\"older regularity then is not difficult
to prove.

Given a finite dimensional sub--space $F$ of $H$ satisfying
conditions \eqref{a:finitespan} and \eqref{a:Fnondegenerate},
an integer $N\geq1$ large enough that $F\subset H_N$, 
and a solution $u^N$ of \eqref{e:galerkin}, define for
every $T>0$ and $\alpha>0$,
\begin{equation}
  \Fc_{F,N}^\alpha(T)
    \eqdef\sup_{t\in[0,T]} t^\alpha\|\ffn(t)\|_\infty.
\end{equation}
\begin{lemma}\label{l:ffinite}
  Under assumptions \eqref{a:diagonal} and \eqref{a:nondegenerate},
  if $N\geq1$ is large enough (that $F\subset H_N$),
  there is $\alpha_0\geq\frac{d}2$ such that
  if $u^N$ is a solution of \eqref{e:galerkin} and
  $\ffn(\cdot)$ is the density of $\pi_Fu^N(\cdot)$,
  then $\Fc_{F,N}^{\alpha_0}(T)<\infty$ for every $T>0$.
\end{lemma}
\begin{proof}
  Fix $k>N$. By Theorem~\ref{t:malliavin} we know that there is
  $\alpha_k>0$ such that
  $M\eqdef\sup_{t,x}(1\wedge t)^{\alpha_k}(1+|x|^k)\fn(x)<\infty$.
  Hence
  \[
    \begin{multlined}[.9\linewidth]
      (1\wedge t)^{\alpha_k}\ffn(t,x')
        = (1\wedge t)^{\alpha_k}\int_\Fperp\fn(t,x)\,dx''\leq\\
        \leq\int_\Fperp\frac1{1+|x''|^k}(1\wedge t)^{\alpha_k}
          (1+|x|^k)\fn(x)\,dx''
        \leq\mconst M.
    \end{multlined}
  \]
  Therefore $\Fc_{F,N}^{\alpha_k}<\infty$.
\end{proof}
The next result provides a uniform (in $N$) estimate
of $\Fc_{F,N}^{d/2}(T)$.
\begin{proposition}\label{p:bounded}
  Under assumptions \eqref{a:diagonal} and \eqref{a:nondegenerate},
  given $T>0$ and $M>0$, if $F$ is a subspace of $H$ satisfying
  \eqref{a:finitespan} and \eqref{a:Fnondegenerate}, there is
  $\mcdef{cc:bounded}=\mcref{cc:bounded}(\nu,T,F,M)>0$
  such that if $N$ is large enough (that $F\subset H_N$), and
  if $u^N$ is a solution of \eqref{e:galerkin} with $\|u^N(0)\|_H\leq M$,
  then
  \[
    \Fc_{F,N}^{d/2}(T)
      \leq \mcref{cc:bounded}.
  \]
\end{proposition}
\begin{proof}
  Fix $T>0$ and a solution $u^N$ of \eqref{e:galerkin} with
  initial condition $x_0$, and set $x_0'=\pi_F x_0$.
  By \eqref{e:FKF}, for every $x'\in F$ and $t\in (0,T]$,
  \begin{equation}\label{e:infty0}
    |\ffn(t,x')|
      \leq \|\pker_t\|_{L^\infty}
        + \int_0^t \|\nabla\pker_{t-s}\star(\gfn(s)\ffn(s))\|_{L^\infty}\,ds.
  \end{equation}
  Fix $\epsilon\in(0,\frac12)$ and let $d=\dim F$.
  On the one hand, by Proposition~\ref{p:preg},
  \begin{equation}\label{e:infty1}
    \begin{aligned}
      \int_0^{\mathrlap{t-t\epsilon}}
          \|\nabla\pker_{t-s}\star(\gfn(s)\ffn(s))\|_{L^\infty}\,ds
        &\leq \int_0^{\mathrlap{t-t\epsilon}}
          \|\nabla\pker_{t-s}\|_{L^\infty}\|(\gfn\ffn)(s)\|_{L^1}\,ds\\
        &\leq \Gc_1(T)\int_0^{t-t\epsilon}
          \frac{\mcref{cc:p}}{(t-s)^{\frac{d+1}2}}\,ds\\
        &\leq \mcdef{cc:infty1}(t\epsilon)^{-\frac12(d-1)}\Gc_1(T).
    \end{aligned}
  \end{equation}
  On the other hand, if $\alpha>0$, $p\in(1,\frac{d}{d-1})$,
  and $q$ is such that
  $\frac1p+\frac1q=1$, by the H\"older inequality and again
  Proposition~\ref{p:preg},
  \begin{equation}\label{e:infty2}
    \begin{aligned}
      \int_{\mathrlap{t-t\epsilon}}^t
           \|\nabla\pker_{t-s}\star(\gfn(s)\ffn(s))\|_{L^\infty}\,ds
        &\leq \int_{\mathrlap{t-t\epsilon}}^t
           \|\nabla\pker_{t-s}\|_{L^p}\|(\gfn\ffn)(s)\|_{L^q}\,ds\\
        &\leq \Fc_{F,N}^\alpha(T)^{\frac1p}\Gc_q(T)\int_{t-t\epsilon}^t
           \frac{\mcref{cc:p}}{s^{\frac{\alpha}p}(t-s)^{\frac{d}{2q}+\frac12}}\,ds\\
        &\leq\mcdef{cc:infty2}\epsilon^{\frac12-\frac{d}{2q}}
           t^{\frac12-\frac\alpha{p}-\frac{d}{2q}}
           \Fc_{F,N}^\alpha(T)^{\frac1p}\Gc_q(T).
    \end{aligned}
  \end{equation}
  Take
  $\beta\geq\frac{d}2\vee\bigl(\frac{d}{2q}+\frac\alpha{p}-\frac12\bigr)$,
  then by using together \eqref{e:infty1} and
  \eqref{e:infty2} into \eqref{e:infty0} we obtain
  \begin{equation}\label{e:infty3}
    \begin{multlined}[.9\linewidth]
    \Fc_{F,N}^\beta(T)
      \leq \mcref{cc:p}T^{\beta-\frac{d}2}
        + \mcref{cc:infty1}\epsilon^{-\frac12(d-1)}
           T^{\beta-\frac12(d-1)} \Gc_1(T) + {}\\
        + \mcref{cc:infty2}\epsilon^{\frac12-\frac{d}{2q}}
           T^{\beta+\frac12-\frac\alpha{p}-\frac{d}{2q}}
           \Fc_{F,N}^\alpha(T)^{\frac1p}\Gc_q(T).
    \end{multlined}
  \end{equation}

  We can now prove that $\Fc_{F,N}^{d/2}(T)<\infty$.
  Indeed, let $\alpha_0$ be the exponent given by
  Lemma~\ref{l:ffinite}, so that $\Fc_{F,N}^{\alpha_0}(T)<\infty$.
  If $\alpha_0=\frac{d}2$ there is nothing to prove. Otherwise we
  can take $\alpha_1=\frac{d}2\vee
  \bigl(\frac{d}{2q}+\frac{\alpha_0}{p}-\frac12\bigr)$
  and, by \eqref{e:infty3}, $\Fc_{F,N}^{\alpha_1}(T)<\infty$
  as well.
  By the choice of $p$ (that gives $q>d$), it turns out
  that $\alpha_1<\alpha_0$.
  It is also clear that by iterating the above procedure
  with exponents $\alpha_2$, $\alpha_3$, \dots,
  in a finite number of steps we will obtain the
  exponent $\frac{d}2$ and thus that $\Fc_{F,N}^{d/2}(T)<\infty$.
  
  Finally, we prove the uniform bound on $\Fc_{F,N}^{d/2}(T)<\infty$.
  Consider again \eqref{e:infty3} with $\alpha=\beta=\frac{d}2$,
  and use the Young inequality,
  \[
    \begin{aligned}
      \Fc_{F,N}^{d/2}(T)
        &\leq \mcref{cc:p}
           + \mcref{cc:infty1}\epsilon^{-\frac12(d-1)}\sqrt{T}\Gc_1(T)
           + \mcref{cc:infty2}\epsilon^{\frac12-\frac{d}{2q}}
               \sqrt{T}\Gc_q(T)\Fc_{F,N}^{d/2}(T)^{\frac1p}\\
        &\leq \mcref{cc:p}
           + \mcref{cc:infty1}\epsilon^{-\frac12(d-1)}\sqrt{T}\Gc_1(T)
           + \frac1q\bigl(\mcref{cc:infty2}\epsilon^{\frac12-\frac{d}{2q}}
               \sqrt{T}\Gc_q(T)\bigr)^q
           + \frac1p\Fc_{F,N}^{d/2}(T).
    \end{aligned}
  \]
  Since $p>1$ we deduce that
  $\Fc_{F,N}^{d/2}(T)\leq C(T,\Gc_1(T),\Gc_q(T))$
  and the conclusion of the theorem finally follows.
\end{proof}
The singularity in time of the $L^\infty$ norm
in the previous result clearly originates
only from the singularity in the initial condition.
It is reasonable then that, when we look for an
initial distribution for $u^N$ with a smoother law,
the same result should hold with a smaller power for
the time. As an example, we give
a direct proof of the $L^\infty$ bound of the density
of the (unique, see \cite{Rom2004}) invariant measure.
Denote by $k_{F,N}$ its density, then as in
Section~\ref{s:formulation}, the density satisfies
the equation
\[
  \frac12\As_F k_{F,N}
      + \Div(\gfn k_{F,N})
    = 0.
\]
Notice that when the averaging in \eqref{e:G} is done with
respect to the invariant measure, $\gfn$ does not depend on
time. Lemma~\ref{l:Gbound} holds though and
$\gfn\in L^p(F;k_{F,N}\,dx')$ for every $p\geq1$.
\begin{proposition}
  Under assumptions \eqref{a:diagonal} and \eqref{a:nondegenerate},
  given a subspace $F$ of $H$ satisfying
  \eqref{a:finitespan} and \eqref{a:Fnondegenerate},
  and $N\geq1$ large enough, let
  $k_{F,N}$ be the density of the unique invariant measure
  of \eqref{e:galerkin}. Then there is $\mcdef{cc:stat}>0$ depending
  only on $\nu$, $F$ and some polynomial moment of the invariant
  measure in $H$ such that $\|k_{F,N}\|_{L^\infty}\leq\mcref{cc:stat}$.
\end{proposition}
\begin{proof}
  Let $d=\dim F$ and denote by $g_F$ the Green function
  of $\As_F$. As in Proposition~\ref{p:preg}, $g_F$ can be
  obtained by an invertible linear transformation from
  the Poisson kernel. Hence $\nabla g_F\leq\mcdef{cc:poisson}|x|^{1-d}$.
  We have
  \[
    \begin{aligned}
      k_{F,N}(x')
        &= - (\nabla g_F)\star(\gfn k_{F,N})(x')\\
        &= -\bigl(\nabla g_F \uno_{B_\epsilon(x')}\bigr)\star(\gfn k_{F,N})(x')
           -\bigl(\nabla g_F \uno_{B_\epsilon^c(x')}\bigr)\star(\gfn k_{F,N})(x')\\
        &= \memo{\,i\,} + \memo{o},
    \end{aligned}
  \]
  where $\epsilon>0$ will be chosen later. Fix $q>d$ and let $p$ be
  its H\"older conjugate exponent. Then
  \[
    |\memo{\,i\,}|
      \leq \|\nabla g_F\|_{L^p(B_\epsilon(0)}\|\gfn k_{F,N}\|_{L^q}
      \leq \mcdef{cc:istat}\epsilon^{\frac{d}{p}-(d-1)}\Gc_q
        \|k_{F,N}\|_{L^\infty}^{1-\frac1q},
  \]
  and likewise
  \[
    |\memo{o}|
      \leq \|\nabla g_F\|_{L^q(B_\epsilon^c(0)}\|\gfn k_{F,N}\|_{L^p}
      \leq \frac{\mcdef{cc:ostat}}{\epsilon^{(d-1) - \frac{d}{q}}}\Gc_p
        \|k_{F,N}\|_{L^\infty}^{1-\frac1p}.
  \]
  where $\Gc_p$, $\Gc_q$ are the quantities in \eqref{e:Gbound}
  but computed on the stationary solution (so there is no need to
  evaluate the supremum in time).

  Collecting the two estimates above and
  choosing $\epsilon = \|k_{F,N}\|_{L^\infty}^{-1/d}$
  yields
  \[
    \begin{aligned}
      |k_{F,N}(x')|
        &\leq \mcref{cc:istat}\epsilon^{\frac{d}{p}-(d-1)}\Gc_q
                \|k_{F,N}\|_{L^\infty}^{1-\frac1q}
              + \frac{\mcref{cc:ostat}}{\epsilon^{(d-1) - \frac{d}{q}}}\Gc_p
                \|k_{F,N}\|_{L^\infty}^{1-\frac1p}\\
        &\leq \mconst(\Gc_q + \Gc_p)\|k_{F,N}\|_\infty^{\frac{d-1}{d}},
    \end{aligned}
  \]
  and hence the statement of the proposition, since $\|k_{F,N}\|_{L^\infty}$
  is non--zero and finite by Lemma~\ref{l:ffinite}.
\end{proof}
Clearly both results above immediately extend to solutions
of the infinite dimensional problem that are limit point of Galerkin
approximations as in Corollary~\ref{c:bbesov}.
\subsection{H\"older regularity}

Using the boundedness of the densities proved so far
we can give an estimate on $\gfn$ different from \eqref{e:Gbound}.
Indeed, if $p>1$ and $t>0$, 
\begin{equation}\label{e:Gboundplus}
  \begin{aligned}
    \|\gfn(t)\ffn(t)\|_{L^p}
      &= \int_F |\gfn(t,x')\ffn(t,x')|^p\,dx'\\
      &\leq \|\ffn(t)\|_{L^\infty}^{p-1}\int_F|\gfn(t,x')|^p\ffn(t,x')\,dx'\\
      &\leq\Gc_p(t)\Fc_{F,N}^{d/2}(t)^{\frac1q}t^{-\frac{d}{2q}},
  \end{aligned}
\end{equation}
where $q$ is the conjugate H\"older exponent of $p$.
In particular it is elementary to verify that
$\gfn\ffn\in L^r([0,T];L^p(F))$ for every
$T>0$ and every $p,r\geq1$, with $\frac2r+\frac{d}{p}>d$.
In other words $\gfn\ffn$ has classical summability.
\begin{proposition}\label{p:holder}
  Under assumptions \eqref{a:diagonal} and \eqref{a:nondegenerate},
  given a subspace $F$ of $H$ satisfying
  \eqref{a:finitespan} and \eqref{a:Fnondegenerate},
  $\alpha\in(0,1)$, $T>0$
  and $M>0$, there is $\mcdef{cc:holder}=
  \mcref{cc:holder}(\nu,\alpha,T,F,M)>0$ such that if $N$ is large
  enough (that $F\subset H_N$), and if $u^N$ is a solution of
  \eqref{e:galerkin} with $\|u^N(0)\|_H\leq M$, then
  \[
    \sup_{t\in [0,T]} t^{\frac12(d+\alpha)}\|\ffn(t)\|_{C^\alpha_b}
      \leq \mcref{cc:holder},
  \]
  where $d=\dim F$ and $\ffn(\cdot)$ is the density of $u^N(\cdot)$.
\end{proposition}
\begin{proof}
  Fix $\alpha\in(0,1)$, $T>0$ and a solution $u^N$ of \eqref{e:galerkin}
  with initial condition $x_0$, and set $x_0'=\pi_F x_0$.
  By \eqref{e:FKF}, for every $x'\in F$, $t\in (0,T]$ and $h\in F$
  with $|h|\leq 1$,
  \[
    \|\Delta_h\ffn(t)\|_{L^\infty}
      \leq \|\Delta_h\pker_t\|_{L^\infty}
        + \int_0^t \|(\Delta_h\nabla\pker_{t-s})\star(\gfn(s)\ffn(s))\|_{L^\infty}\,ds.
  \]
  Fix $\epsilon\in(0,\frac12]$
  and set $q=\frac{d}{1-\alpha}$.
  On the one hand, by Proposition~\ref{p:preg},
  \[
    \begin{aligned}
      \int_0^{\mathrlap{t-t\epsilon}}
          \|(\Delta_h\nabla\pker_{t-s})\star(\gfn(s)\ffn(s))\|_{L^\infty}\,ds
        &\leq \int_0^{\mathrlap{t-t\epsilon}}
          \|\Delta_h\nabla\pker_{t-s}\|_{L^\infty}\|(\gfn\ffn)(s)\|_{L^1}\,ds\\
        &\leq \Gc_1(T)|h|^\alpha\int_0^{t-t\epsilon}
          \frac{\mcref{cc:p}}{(t-s)^{\frac{d+1+\alpha}2}}\,ds\\
        &\leq \mcdef{cc:holder1}(t\epsilon)^{-\frac{d+\alpha-1}2}\Gc_1(T)|h|^\alpha.
    \end{aligned}
  \]
  On the other hand, by the H\"older inequality, the estimate
  \eqref{e:Gboundplus} and again Proposition~\ref{p:preg},
  \[
    \begin{multlined}[.95\linewidth]
      \int_{\mathrlap{t-t\epsilon}}^t
          \|(\Delta_h\nabla\pker_{t-s})\star(\gfn\ffn\|_{L^\infty}\,ds
        \leq \int_{\mathrlap{t-t\epsilon}}^t
          \|\Delta_h\nabla\pker_{t-s}\|_{L^p}\|\gfn\ffn\|_{L^q}\,ds\leq\\
        \leq \Fc_{F,N}^{d/2}(T)^{\frac1p}\Gc_q(T)\int_{t-t\epsilon}^t
          \frac{\mcref{cc:p}}{s^{\frac{d}{2p}}(t-s)^{\frac{d}{2q}+\frac12}}
          \Bigl(1\wedge\frac{|h|}{\sqrt{t-s}}\Bigr)\,ds\leq\\
        \leq\mcdef{cc:holder2}t^{-\frac{d}2(1-\frac1q)}
          \Fc_{F,N}^{d/2}(T)^{\frac1p}\Gc_q(T)|h|^{1-\frac{d}{q}},
    \end{multlined}
  \]
  where $p$ is the conjugate H\"older exponent of $q$.
  Since by definition $1-\frac{d}{q}=\alpha$, the
  two estimates above together yield
  \[
    \begin{multlined}[.9\linewidth]
      \|\Delta_h\ffn(t)\|_{L^\infty}
        \leq \mcref{cc:p}t^{-\frac12(d+\alpha)}|h|^\alpha
          + \mcref{cc:holder1}(t\epsilon)^{-\frac12(d+\alpha-1)}
              \Gc_1(T)|h|^\alpha + {}\\
          + \mcref{cc:holder2}t^{-\frac12(d+\alpha-1)}\Gc_q(T)
              \Fc_{F,N}^{d/2}(T)^{\frac1p}|h|^\alpha,
    \end{multlined}
  \]
  and therefore
  \[
    \sup_{[0,T]}(t^{\frac12(d+\alpha)}[\ffn]_{B^\alpha_{\infty,\infty}})
      \leq \mconst\bigl(1 + \sqrt{T}\Gc_1(T)+\sqrt{T}
        \Gc_{\frac{d}{1-\alpha}}(T)
        \Fc_{F,N}^{d/2}(T)^{\frac{d+\alpha-1}{d}}\bigr).
  \]
  The conclusion follows since $B^\alpha_{\infty,\infty}=C_b^\alpha$.
\end{proof}
\begin{proof}[Proof of Theorem~\ref{t:main}]
    Let $u$ be a solution of \eqref{e:nseabs} according to
    Definition~\ref{d:sol}. Then $u$ is a limit point,
    in distribution, of the sequence $(u^N)_{N\geq1}$
    of solutions of \eqref{e:galerkin}, and the density
    $\ff$ of $\pi_F u$ is a limit point for the weak
    convergence in $L^1$ of $(\ffn)_{N\geq1}$
    by Corollary~\ref{c:bbesov}.

    Proposition~\ref{p:holder} above states that the
    sequence $(\ffn)_{N\geq1}$ is bounded in $C_b^\alpha(F)$,
    hence by the Ascoli--Arzelà theorem each sub--sequence
    admits a further sub--sequence converging uniformly on
    compact sets of $F$ (and weakly in $L^1$ since it is also
    uniformly integrable) to a $C_b^\alpha(F)$ function.
    This proves the theorem.
\end{proof}
\appendix
\section{Technical estimates}

In this appendix we derive some (non--uniform in $N$) quantitative bounds
on the density $\fn$ of $u^N$ that are necessary as an intermediate step
in the proof of our main result. To this end we recall that the solutions
of the Galerkin approximations \eqref{e:galerkin} satisfy
\begin{equation}\label{e:polymoment}
  \E\Bigl[\sup_{[0,T]}\|u^N(t)\|_H^{2p}
      + \nu\int_0^T \|u^N\|_V^2\|u^N\|_H^{2p-2}\,dt\Bigr]
    \leq \mcdef{cc:poly}(p,\nu,T,\|u^N(0)\|_H,\cov),
\end{equation}
for every $N\geq1$, $u^N(0)\in H$, $T>0$ and $p\geq1$, and the number
on the right hand side depends on $N$ only through $\|u^N(0)\|_H$.
In fact there is a stronger estimate: there is $\lambda>0$ such
that for every $N\geq1$, $u^N(0)\in H$ and $T>0$,
\begin{equation}\label{e:expomoment}
  \E\Bigl[\exp\Bigl(\lambda\sup_{[0,T]}\|u^N\|_H^2
      + \nu\lambda\int_0^T\|u^N\|_V^2\,ds\Bigl)\Bigr]
    \leq\mcdef{cc:expo}(\nu,T,\|u^N(0)\|_H,\cov),
\end{equation}
See for instance \cite{Fla2008} for details.

The main point in our estimates of Section~\ref{s:holder} is that,
in order to prove uniform (in $N$) bounds in $L^\infty$ for the
marginal density $\ffn$, we already need to know that $\ffn$ is regular.
This in principle does not follow immediately from
the regularity of $\fn$ (see Example~\ref{x:needschwartz} below)
unless the joint density $\fn$ is in the Schwartz space, as shown in
the result below.
\begin{theorem}\label{t:schwartz}
  Under assumptions \eqref{a:diagonal} and \eqref{a:nondegenerate},
  for every $x_0\in H_N$ the density
  $\fn$ of the solution $u^N$ of \eqref{e:galerkin}, with initial
  condition $x_0$, is in the Schwartz space.
\end{theorem}
\begin{proof}
  The conclusion follows immediately from \cite[Proposition 2.1.5]{Nua2006}
  once we show that $u^N(t)$ is a non--degenerate random variable
  (see \cite[Definition 2.1.1]{Nua2006}), namely that
  $u^N(t)\in\mathbb{D}^\infty$, where $\mathbb{D}^\infty$ is the space
  of random variable with Malliavin derivatives of arbitrary order,
  and that $(\det\mathcal{M}_{N,t})^{-1}\in L^p$ for every $p\geq1$,
  where $\mathcal{M}_{N,t}$ is the Malliavin matrix of $u^N(t)$.
  
  The fact that $u^N(t)\in\mathbb{D}^\infty$ for every $t>0$ can be
  proved as in \cite[Theorem 2.2.2]{Nua2006}. The quoted theorem
  requires that the coefficients should have bounded derivatives
  of any order larger or equal than one. The coefficients of our
  equation \eqref{e:galerkin} are second order polynomials, and
  it is not difficult to verify, by going through the proof
  of the quoted theorem, that the boundedness condition can be
  replaced by \eqref{e:polymoment} to prove existence and
  uniqueness, and by \eqref{e:expomoment} to prove
  existence of the Malliavin derivatives, since the equations
  for the Malliavin derivatives are linear.
   
  Finally, the fact that the inverse of the Malliavin matrix
  has all polynomial moments finite
  follows as in \cite[Theorem 2.3.3]{Nua2006} (where again
  we use \eqref{e:expomoment} to replace the boundedness
  condition on the derivatives), once the H\"ormander condition
  (see \cite[Section 2.3.2]{Nua2006}) holds. Under
  assumptions \eqref{a:diagonal} and \eqref{a:nondegenerate}
  the H\"ormander condition follows from \cite[Lemma 4.2]{Rom2004}.
\end{proof}
From the previous theorem we immediately deduce the following
consequence, whose proof is based on elementary computations.
Notice that, in view of Example~\ref{x:needschwartz},
the fact that the density $\fn$ is in the Schwartz space
is crucial for the proof of boundedness of $\ffn$.
\begin{corollary}
  Under the same assumptions of the previous theorem, the marginal
  density $\ffn$ is the Schwartz space.
\end{corollary}
\begin{example}\label{x:needschwartz}
  It is easy to construct a density $\phi\in C^\infty(\R^d)
  \cap L^\infty(\R^d)$ such that at least one of its marginals
  is not bounded. Indeed, take $d=2$ and let $\varphi\in C^\infty(\R^2)$
  such that $\varphi\geq0$, $\supp\varphi\subset(-\frac12,\frac12)^2$,
  $\|\varphi\|_{L^1}=1$, and
  $\|\varphi(0,\cdot)\|_{L^1} + \|\varphi(\cdot,0)\|_{L^1}>0$.
  Set for every $k\in\Z^2$,
  \[
    \phi_k(x) =
      \begin{cases}
        \varphi(k_1k_2x), \qquad& k_1\neq0,k_2\neq0,\\
        0, &\text{otherwise},
      \end{cases}
  \]
  and $\phi = \sum_k \phi_k(x-k)$. Then $\phi$ is a smooth,
  bounded density (up to re--normalization), but both marginals
  are unbounded.
\end{example}
We give a more quantitative version of the previous theorem by
giving the explicit dependence of the supremum norm of the
density with respect to time. This is necessary for the proof
of Proposition~\ref{p:bounded}. Similar bounds can be also
given for the derivatives of the densities, although we do
not need these estimates in the article.
\begin{theorem}\label{t:malliavin}
  Under assumptions \eqref{a:diagonal} and \eqref{a:nondegenerate},
  let $x_0\in H_N$ and let $\fn$ be
  the density of the solution $u^N$ of \eqref{e:galerkin} with initial
  condition $x_0$. Then for every $k\geq1$ there is $\alpha_k>0$
  such that for every $T>0$,
  \[
    \sup_{t\in(0,T),x\in H_N}\bigl((1\wedge t)^{\alpha_k}(1+|x|^k)\fn(x)\bigr)
      <\infty.
  \]
\end{theorem}
\begin{proof}
  By the previous theorem we know that $\fn$ is in the Schwartz space,
  so we only have to understand the singularity at $t=0$.
  By~\cite[Proposition 2.1.5]{Nua2006} we have the following formulas
  for the density,
  \[
    \fn(t,x)
      = \E\Bigl[\prod_{j=1}^{D_N}\uno_{\{u_{N,j}(t)>x_j\}}H_{N,t}\Bigr]
      = \E\Bigl[\prod_{j\neq i}\uno_{\{u_{N,j}(t)>x_j\}}
          \uno_{\{u_{N,i}(t)<x_i\}}H_{N,t}\Bigr],      
  \]
  for $x\in H_N$ and $t>0$,
  where $D_N=\dim H_N$, $H_{N,t}\in\mathbb{D}^\infty$ is a suitable
  random variable arising from integration by parts, and
  $(u_{N,j})_{j=1,\dots,D_N}$ are the components of $u^N(t)$ with
  respect to a basis of $H_N$. By the Cauchy--Schwartz inequality,
  the representation formulas above and \eqref{e:polymoment},
  \[ 
    \begin{multlined}[.9\linewidth]
      (1 + |x|^k)\fn(x)
        \leq\mconst\E[(1 + \|u^N(t)\|_H^k)|H_t|]\leq\\
        \leq\mconst\E[(1 + \|u^N(t)\|_H^k)^2]^{\frac12}\E[|H_t|^2]^{\frac12}
        \leq\mconst\E[|H_t|^2]^{\frac12},
    \end{multlined}
  \]
  for every $k\geq 1$. By \cite[formula (2.32)]{Nua2006},
  given $p\geq1$ there are $\beta,\gamma>1$ and integers $n,m\geq1$
  such that
  \[
    \E[|H_t|^2]^{\frac1p}
      \leq \mconst\|\det\Mac_{N,t}^{-1}\|_{L^\beta(\Omega)}^m
        \|\Dc u^N(t)\|_{D_N,\gamma},
  \]
  where $\Dc u^N(t)$ and $\Mac_{N,t}$ are respectively the Malliavin
  derivative and the Malliavin matrix of $u^N(t)$, and
  $\|\cdot\|_{k,p}$, for $k\geq1$ and $p\geq1$ is the norm
  of the Malliavin--Sobolev space $\mathbb{D}^{k,p}$,
  see \cite[Section 1.2]{Nua2006}.
  By \cite[Theorem 2.2.2]{Nua2006}, together with our exponential bound
  \eqref{e:expomoment}, as in the proof of the previous theorem,
  we see that
  \[
    \sup_{[0,T]}\|\Dc u^N(t)\|_{D_N,\gamma}
      <\infty,
  \]
  and that the singularity at $t=0$ originates, not unexpectedly,
  from the inverse Malliavin matrix, or equivalently, from the
  smallest eigenvalue $\lambda_{N,t}$ of $\Mac_{N,t}$. Indeed,
  $|\det\Mac_{N,t}^{-1}|\leq \lambda_{N,t}^{-D_N}$, hence
  \[
    \|\det\Mac_{N,t}^{-1}\|_{L^\beta(\Omega)}^m
      \leq \E[\lambda_{N,t}^{-\beta D_N}]^{\frac{m}\beta}.
  \]
  To conclude the proof, we briefly go through the proof of
  \cite[Theorem 2.3.3]{Nua2006} and point out only, for brevity,
  the arguments where the dependence in time arises. To show the
  finiteness of the moments of $\lambda_{N,t}^{-1}$, we look for
  the tail probabilities $\Prob[\lambda_{N,t}\leq\epsilon]$
  at $0$. By \cite[Lemma 2.3.1]{Nua2006}
  \[
    \Prob[\lambda_{N,t}\leq\epsilon]
      \leq \mconst\epsilon^{-2D_N}\sup_{|v|=1}\Prob[v^T\Mac_{N,t}v\leq\epsilon]
        + \Prob[|\Mac_{N,t}|\geq\epsilon^{-1}].
  \]
  The second term on the right hand side is not the source of singularity
  in time and can be easily bounded since $\Mac_{N,t}$ has all the polynomial
  moments finite. The first term is bounded using the Norris lemma
  (\cite[Lemma 2.3.2]{Nua2006}). Indeed for every $q\geq2$ there is
  $\mcdef{cc:norris}>0$ such that
  \[
    \sup_{|v|=1}\Prob[v^T\Mac_{N,t}v\leq\epsilon]
      \leq \mcref{cc:norris}\epsilon^q,
  \]
  if $\epsilon\leq\epsilon_0$, for a suitable $\epsilon_0>0$.
  The value of $\epsilon_0$ is identified through the aforementioned
  Norris lemma, and here one has to look for the singularity in time.
  It is elementary to see that, just by going through the proof of the
  lemma, that $\epsilon_0\approx(1\wedge t)^\alpha$, for some $\alpha$
  that depends on $q$. By this it follows that
  \[
    \E[\lambda_{N,t}^{-\beta D_N}]
      \leq\mconst(1\wedge t)^{-\alpha\beta D_N},
  \]
  as needed.
\end{proof}
\bibliographystyle{amsalpha}

\begin{thebibliography}{HKHY14}

\bibitem[BC14]{BalCar2014p}
Vlad Bally and Lucia Caramellino, \emph{Convergence and regularity of
  probability laws by using an interpolation method}, 2014, \arxiv{1409.3118}
  [math.AP].

\bibitem[Deb13]{Deb2013}
Arnaud Debussche, \emph{Ergodicity results for the stochastic
  {N}avier--{S}tokes equations: an introduction}, Topics in mathematical fluid
  mechanics, Lecture Notes in Math., vol. 2073, Springer, Heidelberg, Berlin,
  2013, Lectures given at the C.I.M.E.--E.M.S. Summer School in applied
  mathematics held in Cetraro, September 6--11, 2010, Edited by Franco Flandoli
  and Hugo Beirao da Veiga, pp.~23--108. \MR{3076070 (2013:)}

\bibitem[DF13]{DebFou2013}
Arnaud Debussche and Nicolas Fournier, \emph{Existence of densities for
  stable-like driven {SDE}'s with {H}\"older continuous coefficients}, J.
  Funct. Anal. \textbf{264} (2013), no.~8, 1757--1778. \MR{3022725 (2013:)}

\bibitem[DM11]{Dem2011}
Stefano De~Marco, \emph{Smoothness and asymptotic estimates of densities for
  {SDE}s with locally smooth coefficients and applications to square root-type
  diffusions}, Ann. Appl. Probab. \textbf{21} (2011), no.~4, 1282--1321.
  \MR{2857449 (2012g:60186)}

\bibitem[DO06]{DebOda2006}
Arnaud Debussche and Cyril Odasso, \emph{Markov solutions for the 3{D}
  stochastic {N}avier-{S}tokes equations with state dependent noise}, J. Evol.
  Equ. \textbf{6} (2006), no.~2, 305--324. \MR{2227699 (2007f:35214)}

\bibitem[DPD03]{DapDeb2003}
Giuseppe Da~Prato and Arnaud Debussche, \emph{Ergodicity for the 3{D}
  stochastic {N}avier-{S}tokes equations}, J. Math. Pures Appl. (9) \textbf{82}
  (2003), no.~8, 877--947. \MR{2005200 (2004m:60133)}

\bibitem[DPZ92]{DapZab1992}
Giuseppe Da~Prato and Jerzy Zabczyk, \emph{Stochastic equations in infinite
  dimensions}, Encyclopedia of Mathematics and its Applications, vol.~44,
  Cambridge University Press, Cambridge, 1992. \MR{1207136 (95g:60073)}

\bibitem[DR14]{DebRom2014}
Arnaud Debussche and Marco Romito, \emph{Existence of densities for the 3{D}
  {N}avier--{S}tokes equations driven by {G}aussian noise}, Probab. Theory
  Related Fields \textbf{158} (2014), no.~3-4, 575--596. \MR{3176359 (2014:)}

\bibitem[Fla08]{Fla2008}
Franco Flandoli, \emph{An introduction to 3{D} stochastic fluid dynamics},
  S{PDE} in hydrodynamic: recent progress and prospects, Lecture Notes in
  Math., vol. 1942, Springer, Berlin, 2008, Lectures given at the C.I.M.E.
  Summer School held in Cetraro, August 29--September 3, 2005, Edited by
  Giuseppe Da Prato and Michael R{\"o}ckner, pp.~51--150. \MR{2459085
  (2009j:76191)}

\bibitem[Fou15]{Fou2015}
Nicolas Fournier, \emph{Finiteness of entropy for the homogeneous {B}oltzmann
  equation with measure initial condition}, Ann. Appl. Probab. \textbf{25}
  (2015), no.~2, 860--897. \MR{3313757 (2015:)}

\bibitem[FP10]{FouPri2010}
Nicolas Fournier and Jacques Printems, \emph{Absolute continuity for some
  one--dimensional processes}, Bernoulli \textbf{16} (2010), no.~2, 343--360.
  \MR{2668905 (2011d:60176)}

\bibitem[FR06]{FlaRom2006}
Franco Flandoli and Marco Romito, \emph{Markov selections and their regularity
  for the three-dimensional stochastic {N}avier-{S}tokes equations}, C. R.
  Math. Acad. Sci. Paris \textbf{343} (2006), no.~1, 47--50. \MR{2241958
  (2007m:60177)}

\bibitem[FR07]{FlaRom2007}
\bysame, \emph{Regularity of transition semigroups associated to a 3{D}
  stochastic {N}avier-{S}tokes equation}, Stochastic differential equations:
  theory and applications (Peter~H. Baxendale and Sergey~V. Lototski, eds.),
  Interdiscip. Math. Sci., vol.~2, World Sci. Publ., Hackensack, NJ, 2007,
  pp.~263--280. \MR{2393580 (2009g:60076)}

\bibitem[FR08]{FlaRom2008}
\bysame, \emph{Markov selections for the 3{D} stochastic {N}avier-{S}tokes
  equations}, Probab. Theory Related Fields \textbf{140} (2008), no.~3-4,
  407--458. \MR{2365480 (2009b:76033)}

\bibitem[HKHY13]{HayKohYuk2013}
Masafumi Hayashi, Arturo Kohatsu-Higa, and G{\^o}~Y{\^u}ki, \emph{Local
  {H}\"older continuity property of the densities of solutions of {SDE}s with
  singular coefficients}, J. Theoret. Probab. \textbf{26} (2013), no.~4,
  1117--1134. \MR{3119987 (2013:)}

\bibitem[HKHY14]{HayKohYuk2014}
\bysame, \emph{H\"older continuity property of the densities of {SDE}s with
  singular drift coefficients}, Electron. J. Probab. \textbf{19} (2014), no.
  77, 22. \MR{3256877 (2014:)}

\bibitem[KHT12]{KohTan2012}
Arturo Kohatsu-Higa and Akihiro Tanaka, \emph{A {M}alliavin calculus method to
  study densities of additive functionals of {SDE}'s with irregular drifts},
  Ann. Inst. Henri Poincar{\'e} Probab. Stat. \textbf{48} (2012), no.~3,
  871--883. \MR{2976567 (2012:)}

\bibitem[MP06]{MatPar2006}
Jonathan~C. Mattingly and {\'E}tienne Pardoux, \emph{Malliavin calculus for the
  stochastic 2{D} {N}avier-{S}tokes equation}, Comm. Pure Appl. Math.
  \textbf{59} (2006), no.~12, 1742--1790. \MR{2257860 (2007j:60082)}

\bibitem[Nua06]{Nua2006}
David Nualart, \emph{The {M}alliavin calculus and related topics}, second ed.,
  Probability and its Applications (New York), Springer-Verlag, Berlin, Berlin,
  2006. \MR{2200233 (2006j:60004)}

\bibitem[Oda07]{Oda2007}
Cyril Odasso, \emph{Exponential mixing for the 3{D} stochastic
  {N}avier-{S}tokes equations}, Comm. Math. Phys. \textbf{270} (2007), no.~1,
  109--139. \MR{2276442 (2008e:60190)}

\bibitem[Rom04]{Rom2004}
Marco Romito, \emph{Ergodicity of the finite dimensional approximation of the
  3{D} {N}avier-{S}tokes equations forced by a degenerate noise}, J. Statist.
  Phys. \textbf{114} (2004), no.~1-2, 155--177. \MR{2032128 (2005a:76128)}

\bibitem[Rom08]{Rom2008}
\bysame, \emph{Analysis of equilibrium states of {M}arkov solutions to the 3{D}
  {N}avier-{S}tokes equations driven by additive noise}, J. Stat. Phys.
  \textbf{131} (2008), no.~3, 415--444. \MR{2386571 (2010a:35200)}

\bibitem[Rom13]{Rom2013a}
\bysame, \emph{Densities for the {N}avier--{S}tokes equations with noise},
  2013, Lecture notes for the ``Winter school on stochastic analysis and
  control of fluid flow'', School of Mathematics of the Indian Institute of
  Science Education and Research, Thiruvananthapuram (India).

\bibitem[Rom14a]{Rom2014b}
\bysame, \emph{Some probabilistic topics in the {N}avier--{S}tokes equations},
  2014, to appear on the proceedings of the workshop ``The Navier--Stokes in
  Venice''.

\bibitem[Rom14b]{Rom2014pb}
\bysame, \emph{Time regularity of the densities for the {N}avier--{S}tokes
  equations with noise}, 2014, \arxiv{1409.1700}.

\bibitem[Rom14c]{Rom2014a}
\bysame, \emph{Unconditional existence of densities for the {N}avier-{S}tokes
  equations with noise}, Mathematical analysis of viscous incompressible fluid,
  RIMS K{\^o}ky{\^u}roku, vol. 1905, Kyoto University, 2014, pp.~5--17.

\bibitem[RX11]{RomXu2011}
Marco Romito and Lihu Xu, \emph{Ergodicity of the 3{D} stochastic
  {N}avier-{S}tokes equations driven by mildly degenerate noise}, Stochastic
  Process. Appl. \textbf{121} (2011), no.~4, 673--700. \MR{2770903
  (2012e:76098)}

\bibitem[SSS15a]{SanSus2015}
Marta Sanz-Sol{\'e} and Andr{\'e} S{\"u}{\ss}, \emph{Absolute continuity for
  {SPDE}s with irregular fundamental solution}, Electron. Commun. Probab.
  \textbf{20} (2015), no. 14, 11. \MR{3314649 (2015:)}

\bibitem[SSS15b]{SanSus2015p}
\bysame, \emph{Non elliptic {SPDE}s and ambit fields: existence of densities},
  2015, \arxiv{1502.02386}.

\bibitem[Tem95]{Tem1995}
Roger Temam, \emph{Navier-{S}tokes equations and nonlinear functional
  analysis}, second ed., CBMS-NSF Regional Conference Series in Applied
  Mathematics, vol.~66, Society for Industrial and Applied Mathematics (SIAM),
  Philadelphia, PA, 1995. \MR{1318914 (96e:35136)}

\bibitem[Tri83]{Tri1983}
Hans Triebel, \emph{Theory of function spaces}, Monographs in Mathematics,
  vol.~78, Birkh\"auser Verlag, Basel, 1983. \MR{781540 (86j:46026)}

\bibitem[Tri92]{Tri1992}
\bysame, \emph{Theory of function spaces. {II}}, Monographs in Mathematics,
  vol.~84, Birkh\"auser Verlag, Basel, 1992. \MR{1163193 (93f:46029)}

\end{thebibliography}

\end{document}